\documentclass[12pt,reqno]{amsart}
\usepackage[margin=1in]{geometry}
\usepackage{graphicx}
\usepackage{url}
\usepackage{multicol}  
\usepackage{amssymb} 
\usepackage[T1]{fontenc}
\usepackage[utf8]{inputenc}
\usepackage[russian,english]{babel}

\theoremstyle{plain}
  \newtheorem{theorem}{Theorem}[section]
  \newtheorem{lemma}[theorem]{Lemma}
  \newtheorem{definition}[theorem]{Definition}
  \newtheorem{corollary}[theorem]{Corollary}
  \newtheorem{problem}[theorem]{Problem}
\theoremstyle{remark}
  \newtheorem{example}[theorem]{Example}
  \newtheorem*{ack}{Acknowledgments}

\newcommand{\pref}[1]{(\ref{#1})}  
\newcommand{\PPP}{\mathsf{P}}
\newcommand{\ZZZ}{\mathbb{Z}}
\newcommand{\SA}{\mathit{SA}}  

\newcommand{\LIP}{\textsc{LIP}}
\newcommand{\RIP}{\textsc{RIP}}
\newcommand{\AAIP}{\textsc{AAIP}}
\newcommand{\FLEX}{\textsc{Flex}}

\newcommand{\RALT}{\textsc{RAlt}}
\newcommand{\LALT}{\textsc{LAlt}}

\newcommand\inv{^{-1}} 
\newcommand\ldv{\backslash} 
\newcommand\rdv{/}  

\title[Loops with Universal and Semi-Universal Flexibility]
{Loops with Universal and Semi-Universal Flexibility}

\author[Britten]{Riley~Britten}
\address{Department of Mathematics \\
University of Denver \\
Denver, CO 80208 USA}
\email{nrb1324@hotmail.com}
\author[Kinyon]{Michael~Kinyon}
\address{Department of Mathematics \\
University of Denver \\
Denver, CO 80208 USA}
\email{michael.kinyon@du.edu}
\author[Kunen]{Kenneth~Kunen}
\address{Department of Mathematics \\
University of Wisconsin \\
Madison, WI 57306 USA}
\author[Phillips]{J.~D.~Phillips}
\address{Department of Mathematics \& Computer Science \\
Northern Michigan University \\
Marquette, Michigan 49855 USA}
\email{jophilli@nmu.edu}

\date{\today}


\begin{document}

\begin{abstract}
We study loops which are universal (that is, isotopically invariant) with
respect to the property of flexibility ($xy\cdot x = x\cdot yx$). We also
weaken this to semi-universality, that is, loops in which every left and
right isotope is flexible, but not necessarily every isotope. One of our main
results is that universally flexible, inverse property loops are Moufang loops.
On the other hand, semi-universally flexible, inverse property loops are diassociative.
We also examine the relationship between universally flexible loops and middle Bol
loops. The paper concludes with some open problems.
\end{abstract}

\maketitle

\begin{center}
\textsc{In memory of Ken Kunen (1943--2020)}
\end{center}


\section{The story of this paper}
\label{sec:story}

This paper was started in 2003 by MK, KK and JDP. Most of the results herein, especially Theorems \ref{thm:main1} and \ref{thm:main2},
were proved in that year; only a few were proved more recently. However, we were unable to settle what we have now stated as
Problem \ref{prb:prob1}.
This was an unsatisfactory state of affairs because if the answer turned out to be negative (which we
suspect is not the case), it would imply that in Theorem \ref{thm:main2}, we simply had a more complicated proof of the diassociativity
of Moufang loops than already existing ones. Thus we set the paper aside in the hope that we could resolve the issue.

Years passed. We occasionally would dig up our notes, think about the problem, get nowhere, and set it aside once again. And of course,
we all got busy with other projects and other priorities.

In 2022, MK suggested to his PhD student RB that searching for an example to resolve Problem \ref{prb:prob1} would make
a nice topic to fit into their dissertation. As discussed in {\S}7, that particular search was unsuccessful, but the
methodology led to other useful examples, especially those appearing in {\S}4 and {\S}5.

Ken Kunen passed away in 2020. It is difficult to overstate the influence that he had on our work. It was thanks to him
that both MK and JDP have focused their research careers on the use of automated deduction tools in quasigroup theory and
other parts of algebra. After mulling it over for a long time, we decided that the best tribute we could give to him
was to finally submit the paper for publication, leaving Problem \ref{prb:prob1} open with the hope that someone will be able
to resolve it. We also note that Ken deserves full coauthor status because all of {\S}6, the proof of Theorem \ref{thm:main2},
was entirely his work.

\section{Introduction}
\label{sec:intro}

A \emph{loop} $(Q,\cdot,1)$ is a set $Q$ with a binary operation $\cdot$ such that there is an identity element $1$ and
for each $a\in Q$, the mappings $L(a):Q\to Q; x\mapsto a\cdot x$ and $R(a):Q\to Q; x\mapsto x\cdot a$ are bijections of $Q$.
An equivalent and also useful definition is: A loop $(Q,\cdot,\ldv,\rdv,1)$ is a set $Q$ with three binary operations
$\cdot, \ldv, \rdv : Q\times Q\to Q$ and an element $1\in Q$ satisfying
the equations:
\begin{align*}
x\ldv (x\cdot y) = y = (y\cdot x) \rdv x \\
x\cdot (x\ldv y) = y = (y\rdv x)\cdot x \\
1\cdot x = x = x\cdot 1\,.
\end{align*}
Basic references for loop theory are \cite{Bel}, \cite{Br}, \cite{Pf}.

For a loop $Q$ and fixed elements $a,b\in Q$, the binary operations defined by
\[
x \cdot_{a,b} y := (x\rdv a)(b\ldv y)  \qquad
x \ldv_{a,b}\ y := b\ [ (x\rdv a)\ldv y ] \qquad
x \rdv_{a,b}\ y := [ x\rdv (b\ldv y) ] a
\]
give $Q$ a new loop structure $(Q, \cdot_{a,b},\ldv_{a,b},\rdv_{a,b}, ba)$ called a (principal) \emph{loop isotope}
of $Q$. \emph{Right} loop isotopes are those with $a=1$ and \emph{left} loop isotopes are those with $b=1$.
Much of loop theory is devoted to the study of \emph{universal} or \emph{isotopically invariant}
properties:

\begin{definition}
If $\PPP$ is a property of loops, then a loop $Q$ is:
\begin{itemize}
\item \emph{left semi-universally} $\PPP$ iff all left loop isotopes satisfy $\PPP$;
\item \emph{right semi-universally} $\PPP$ iff all right loop isotopes satisfy $\PPP$;
\item \emph{semi-universally} $\PPP$ if it is both left and right semi-universally $\PPP$;
\item \emph{universally} $\PPP$ iff all loop isotopes satisfy $\PPP$.
\end{itemize}
\end{definition}

``Universally $\PPP$'' may or may not be strictly stronger than ``semi-universally $\PPP$'', which may or may not be
strictly stronger than ``$\PPP$''. For example, all Moufang loops are universally Moufang. Moufang loops are alternative
and satisfy the inverse property (IP). Every semi-universally IP loop is Moufang, and every semi-universally alternative
loop is Moufang. Thus, ``universally IP'' $=$ ``semi-universally IP'' is strictly stronger than ``IP'', and
``universally alternative'' $=$ ``semi-universally alternative'' is strictly stronger than ``alternative''.

We interject here the definitions of the above terms, plus a few others used in the sequel:

A loop is \emph{Moufang} iff it satisfies the following four equivalent identities:
\[
\begin{array}{rccrc}
M1: & (x\cdot yz)x = xy\cdot zx & \qquad & M2: & x(yz\cdot x) = xy\cdot zx \\
N1: & (xy\cdot z)y = x(y\cdot zy) & \qquad & N2: & x(y\cdot xz) = (xy\cdot x)z
\end{array}
\]

For an element $x$ of a loop, if $1\rdv x = x\ldv 1$, their common value is denoted by $x\inv$.
In a loop in which $1\rdv x = x\ldv 1$ holds for all $x$,
the \emph{left inverse property} (\LIP), the \emph{right inverse property} (\RIP), and the
\emph{antiautomorphic inverse property} (\AAIP) are defined by the identities
\begin{alignat*}{3}
 x\ldv y &= x\inv y &&\text{ or } L(x)\inv = L(x\inv) \tag{\LIP} \\
 x\rdv y &= xy\inv &&\text{ or } R(y)\inv = R(y\inv) \tag{\RIP} \\
 (xy)\inv &= y\inv x\inv\,. && \tag{\AAIP}
\end{alignat*}
Any two of these identities implies the third, and a loop satisfying all three is said to have
the \emph{inverse property} (IP). (In fact, all three properties can be defined in such a way
that the property $1\rdv x = x\ldv 1$ is a theorem, but the equivalent definition above is more
convenient for our purposes.)

Note that every AAIP loop, and in particular every IP loop, is isomorphic to its opposite loop,
so whenever we know that an identity holds in such a loop, we automatically
have the mirror of that identity.

A loop is \emph{diassociative} iff any two of its elements generate a group, and
\emph{power-associative} iff any one of its elements generate a group.
Three well-known weakenings of diassociativity are the
\emph{left alternative property} (\LALT), the
\emph{right alternative property} (\RALT), and the
\emph{flexible property} (\FLEX) defined by
\begin{alignat*}{3}
x\cdot xy &= x^2 y &&\text{ or } L(x^2) = L(x)^2 \tag{\LALT} \\
xy\cdot y &= xy^2  &&\text{ or }  R(y^2) = R(y)^2 \tag{\RALT} \\
x\cdot yx &= xy\cdot x && \text{ or } [R(x), L(x)] = I \tag{\FLEX}
\end{alignat*}
The term ``flexible'' has been translated into Russian as
``\foreignlanguage{russian}{эластичны}''
and translated back into English as ``elastic''. In a flexible loop,
we have $1\rdv x = x\ldv 1$ for all $x$. A loop is \emph{alternative} iff it satisfies both \LALT\, and \RALT.
As usual in loop theory, we often write such equations as statements about the left and right multiplications,
since proofs are often simpler when expressed in terms of some permutation group of the loop.

A \emph{left Bol loop} is one satisfying the identity $x(y\cdot xz) = (x\cdot yx)z$ or equivalently,
$L(x\cdot yx) = L(x)L(y)L(x)$. All left Bol loops are power associative and satisfy the \LIP\, and \LALT. All left Bol
loops are universally left Bol, and any loop which is either (semi-)universally \LIP\, or (semi-)universally \LALT\, is left Bol.
The opposite loop of a left Bol loop is a \emph{right Bol loop} defined by the identity
$(xy\cdot z)y = x(yz\cdot y)$ or $R(yz\cdot y) = R(y)R(z)R(y)$; so loops which are
(semi-)universally RIP or (semi-)universally RAP are right Bol.

A loop which is universally AAIP is called a \emph{middle Bol loop}; these are characterized by the identity
$x((yz)\ldv x) = (x\rdv z)(y\ldv x)$ or equivalently, the mirror of this identity. The operation $x\circ y := x\rdv y\inv$ in a
left Bol loop defines a middle Bol loop. A loop is Moufang iff it satisfies any two, and hence all three, of the
Bol properties.

Another interesting variety properly containing the Moufang loops consists of loops which are universally flexible (UF);
these have been studied in \cite{RR, Sh1, Sh2, Sh3, Sh4, Syr1, Syr2}.
To express this property as an equation, we expand
$x \cdot_{u,v} (y\cdot_{u,v} x) = (x \cdot_{u,v} y)\cdot_{u,v} x$
in terms of the left and right division operations to get the following three equivalent equations:
\begin{eqnarray}
\label{eq:uflex}
(x\rdv u)[v\ldv ((y\rdv u)(v\ldv x))] &=& [((x\rdv u)(v\ldv y))\rdv u] (v\ldv x) \\
\label{eq:uflex-t}
R(u)\inv R(v\ldv x) L(v)\inv L(x\rdv u)& =& L(v)\inv L(x\rdv u) R(u)\inv R(v\ldv x) \\
\label{eq:uflex-com}
[R(u)\inv R(v\ldv x),\;  L(v)\inv L(x\rdv u)] &=& I \ \ .
\end{eqnarray}
Such loops need not have the IP, but in \cite{Syr1, Syr2}, Syrbu studied UF IP loops, showing that they share
many properties with Moufang loops. This turns out not to be surprising
in view of our first result:

\begin{theorem}
\label{thm:main1}
Every universally flexible, IP loop is Moufang.
\end{theorem}

The equations defining the variety of semi-universally flexible (SUF) loops are:
\begin{eqnarray}
x (v\ldv [y (v\ldv x)]) &=& [x (v\ldv y)] (v\ldv x)
\label{eq:hflex1} \\
R(v\ldv x) L(v)\inv L(x) &=& L(v)\inv L(x) R(v\ldv x)
\label{eq:hflex-t1} \\
&\text{and}&
\nonumber \\
(x\rdv u)[(y\rdv u) x] &=& ([(x\rdv u) y]\rdv u) x
\label{eq:hflex2} \\
R(u)\inv R(x) L(x\rdv u) &=& L(x\rdv u) R(u)\inv R(x)
\label{eq:hflex-t2}
\end{eqnarray}

Moufang's Theorem implies that Moufang loops are diassociative. This has been generalized \cite{KKP} to a
wider class of loops called \emph{ARIF} loops, which are flexible loops satisfying the (equivalent) identities
\[
R(xyx)R(y) = R(x)R(yxy) \qquad L(xyx)L(y) = L(x)L(yxy)
\]
In this paper we offer a (possibly) different generalization:

\begin{theorem}
\label{thm:main2}
Every SUF IP loop is diassociative.
\end{theorem}

In fact, this theorem will turn out to be a corollary of a theorem asserting that every loop in an even wider class,
which we call \emph{stepping up}, is diassociative.

\section{Semi-universal flexibility}
\label{sec:leftSUF}

The variety of left SUF loops is defined by the identity \eqref{eq:hflex2}.
Replacing $x$ with $xu$ and $y$ with $yu$ and then rearranging, we obtain
\begin{equation}\label{eq:leftSUF3}
(x\cdot yu)\rdv u = (x (y\cdot xu))\rdv (xu)\,.
\end{equation}
This equation implies the following.

\begin{lemma}\label{lem:assoc}
For $a,b,c$ in a left SUF loop $Q$, the following are equivalent:
(i) $a\cdot bc = ab\cdot c$; (ii) $ab\cdot ac = a(b\cdot ac)$; (iii) $ab\cdot (a\ldv c) = a(b(a\ldv c))$.
\end{lemma}

We will soon show that left SUF loops are power-associative. Until then, we use left-association
to define powers of elements: $x^n := 1 L(x)^n$ for all $n\in \ZZZ$.

\begin{lemma}\label{lem:xyx}
For all $x,y$ in a left SUF loop $Q$ and for all $n\in \ZZZ$, $x\cdot yx^n = xy\cdot x^n$.
\end{lemma}
\begin{proof}
From Lemma~\ref{lem:assoc}, for each $n\in \ZZZ$,
$x\cdot yx^n = xy\cdot x^n$ iff $x\cdot yx^{n+1} = xy\cdot x^{n+1}$ iff $x\cdot yx^{n-1} = xy\cdot x^{n-1}$.
The case $n=0$ being trivial, an induction argument completes the proof.
\end{proof}

\begin{theorem}\label{thm:pa}
Left SUF loops are power-associative.
\end{theorem}
\begin{proof}
Let $Q$ be a left SUF loop. We will prove $x^m x^n = x^{m+n}$ for all $m,n\in \ZZZ$.
For $m = 0$, there is nothing to prove. Taking $y = x^m$ in Lemma~\ref{lem:xyx}, we have
\begin{equation}\label{eq:pa-tmp}
x\cdot x^m x^n = x^{m+1} x^n .
\end{equation}
Assume $x^m x^n = x^{m+n}$ for some $m,n$.
Then $x^{m+n+1} = x x^{m+n} = x\cdot x^m x^n = x^{m+1} x^n$ by \eqref{eq:pa-tmp}.
By induction, we have the desired result for all nonnegative $m$. Replacing $m$ with $-m$,
we can then induct again to complete the argument.
\end{proof}

Now we work in SUF loops, that is, loops that are both left and right SUF. We will use
power-associativity without explicit reference. We also have
available both Lemma~\ref{lem:assoc} and its dual, so that for $a,b,c$ in an SUF loop $Q$,
\begin{equation}\label{eqn:abc}
\begin{aligned}
ab\cdot c = a\cdot bc &\iff ab\cdot ac = a(b\cdot ac) &&\iff (ac\cdot b)c = ac\cdot bc \\
&\iff ab\cdot (a\ldv c) = a(b(a\ldv c)) &&\iff ((a\rdv c)b)c = (a\rdv c)\cdot bc
\end{aligned}
\end{equation}

%

\begin{theorem}\label{thm:xyx}
Let $Q$ be an SUF loop. For all $x,y\in Q$ and
for all $m,n\in \ZZZ$, $x^m y\cdot x^n = x^m \cdot y x^n$, that is, $L(x^m) R(x^n) = R(x^n) L(x^m)$.
\end{theorem}
\begin{proof}
For $m=0$ or $n=0$, there is nothing to prove. By repeated use of \eqref{eqn:abc}, we have that
$x^m y\cdot x^n = x^m\cdot yx^n$ iff $x^m y\cdot x^{im+n} = x^m\cdot yx^{im+n}$ iff
$x^{m+j(im+n)} y\cdot x^n = x^{m+j(im+n)}\cdot yx^n$ for all $i,j\in \ZZZ$. For given $m,n$, we may choose $i$
so that $im+n > 0$ and then choose $j$ so that $m+j(im+n) >  0$. Thus it is sufficient to consider the case
$m,n > 0$. Flexibility is the subcase $m=n$, so assume that $n > m$. Write $n = qm + r_1$ for some
$q,r_1\in \ZZZ$ with $0\leq r_1 < m$. By repeated application of \eqref{eqn:abc},
$x^m y\cdot x^n = x^m \cdot y x^n$ iff $x^m y\cdot x^{r_1} = x^m\cdot y x^{r_1}$. Repeating this argument, there
exists $0\leq r_2 < r_1$ such that $x^m y\cdot x^n = x^m \cdot y x^n$ iff $x^{r_2}y\cdot x^{r_1} = x^{r_2}\cdot yx^{r_1}$.
We can thus eventually reduce the exponents to the case covered by Lemma~\ref{lem:xyx} or its dual. Thus
$x^m y\cdot x^n = x^m\cdot yx^n$ holds. The case $m > n$ is proven similarly.
\end{proof}

\begin{example}
Table \ref{table:lsuf} is a left semi-universally flexible loop that is
not right semi-universally flexible. Here
$1 (3\ldv [2 (3\ldv 1)]) \neq [1 (3\ldv 2)] (3\ldv 1)$.
This loop is simple and has trivial nuclei. In fact, its only proper
subloops are the two-element subloops. This loop satisfies \LIP\, and \LALT\,
(which are the same property here). In addition, every mapping $R(x) L(x)^{-1}$
is an automorphism; this property is stronger than flexibility.
It is easy to check that none of the loops of order 5 have proper
left semi-universal flexibility, and so this example is the smallest such
loop. A computer search using \textsc{Mace4}\, shows that this is the only example
of order 6 up to isomorphism.

\begin{table}[htb]\label{table:lsuf}
\[
\begin{array}{c|cccccc|}
\cdot &  0 & 1 & 2 & 3 & 4 & 5 \\
\hline
  0 &  0& 1& 2& 3& 4& 5 \\
  1 &  1& 0& 3& 2& 5& 4 \\
  2 &  2& 5& 0& 4& 3& 1 \\
  3 &  3& 4& 5& 0& 1& 2 \\
  4 &  4& 2& 1& 5& 0& 3 \\
  5 &  5& 3& 4& 1& 2& 0 \\
\hline
\end{array}
\]
\caption{A left semi-universally flexible loop that is not right semi-universally flexible}
\end{table}
\end{example}

\section{UF IP loops are Moufang}
\label{sec:basics}

In a right SUF LIP loop, we may replace $v$ with $v\inv$ in \eqref{eq:hflex1} to obtain
\begin{align}
x\cdot v(y\cdot vx) &= (x\cdot vy) \cdot vx \label{eq:IPhflex1} \\
R(vx) L(v) L(x) &= L(v) L(x) R(vx) \label{eq:IPhflex-t1}
\end{align}
Similarly, in a left SUF RIP loop satisfying, replacing $u$ with $u\inv$ in \eqref{eq:hflex2} gives
\begin{align}
xu\cdot (yu\cdot x) = (xu\cdot y)u\cdot x \label{eq:IPhflex2} \\
R(u) R(x) L(xu) = L(xu) R(u) R(x) \label{eq:IPhflex-t2}
\end{align}
A loop satisfying \eqref{eq:IPhflex1} is clearly flexible (take $v=1$). Taking $y = v\inv$ in \eqref{eq:IPhflex1},
we obtain $x\cdot v (v\inv\cdot vx) = x \cdot vx$. Cancelling twice,
we get $v\inv\cdot vx = x$, which is the LIP.
We have proved

\begin{lemma}\label{lem:equiv1}
{\ }
\begin{enumerate}
\item A loop satisfies \eqref{eq:IPhflex1} if and only if it is a left SUF LIP loop.
\item A loop satisfies \eqref{eq:IPhflex2} if and only if it is a right SUF RIP loop.
\end{enumerate}
\end{lemma}

\begin{lemma}\label{lem:sufip-alt}
Every SUF IP loop is alternative.
\end{lemma}
\begin{proof}
Replace $v$ with $yx\inv$ in \eqref{eq:IPhflex-t1} to get
\begin{equation}\label{eqn:sufip-alt-tmp1}
R(y) L(yx\inv) L(x) = L(yx\inv) L(x) R(y)\,.
\end{equation}
In this, replace $x$ with $xy$ to get
\begin{equation}\label{eqn:sufip-alt-tmp2}
R(y) L(x\inv) L(xy) = L(x\inv) L(xy) R(y)\,.
\end{equation}
Apply both \eqref{eqn:sufip-alt-tmp1} and \eqref{eqn:sufip-alt-tmp2} to the identity element $1$ to obtain,
respectively,
\begin{align*}
  x\cdot yx\inv y &= xyx\inv\cdot y \\
  xy\cdot x\inv y &= xyx\inv\cdot y\,.
\end{align*}
Thus $x\cdot yx\inv y = xy\cdot x\inv y$. By Lemma~\ref{lem:assoc} and the LIP, $x\cdot yy = xy\cdot y$,
that is \RALT\, holds. Since $Q$ has the IP, \LALT\, also holds.
\end{proof}

In a UF IP loop, we may replace $u$ with $u\inv$ and $v$ with $v\inv$ in \eqref{eq:uflex}, to obtain
\begin{align}
xu\cdot v(yu\cdot vx) &= (xu\cdot vy)u\cdot vx \label{eq:IPuflex} \\
R(u) R(vx) L(v) L(xu) &= L(v) L(xu) R(u) R(vx) \label{eq:IPuflex-t}
\end{align}
Replacing $v$ with $vx\inv$, $u$ with $x\inv u$, and then $x$ with $x\inv$ gives:
\begin{align}
u [vx\cdot (y\cdot xu)v] &= [u(vx\cdot y)\cdot xu]v \label{eq:IPuflex2} \\
R(xu) R(v) L(vx) L(u) &= L(vx) L(u) R(xu) R(v) \label{eq:IPuflex-t2}
\end{align}

\begin{lemma}\cite{Syr1}
\label{lem:equiv2}
A loop satisfies \eqref{eq:IPuflex} or \eqref{eq:IPuflex2} if and only
if it is a universally flexible, IP loop.
\end{lemma}

\begin{proof}
We have established the ``if" direction. If, say, \eqref{eq:IPuflex2} holds,
then taking $u = 1$ gives \eqref{eq:IPhflex2} (after relabelling), and taking
$v = 1$ gives \eqref{eq:IPhflex1}. Then apply Lemma~\ref{lem:equiv1}.
\end{proof}

\begin{proof}[Proof of Theorem~\ref{thm:main1}]
We use the IP and alternative properties
without comment. Take $v = x\inv$ in \eqref{eq:IPuflex-t2}
and rearrange to get
\[
R(x\inv ) L(u) R(x) = R(u\inv x\inv ) L(u) R(xu)
\]
Multiply on the right by $R(v)$ and apply \eqref{eq:IPuflex-t2} again
to obtain
\begin{eqnarray*}
R(x\inv ) L(u) R(x) R(v) &=& R((xu)\inv ) L((vx)\inv ) [ L(vx) L(u) R(xu) R(v) ] \\
&=& R((xu)\inv ) L((vx)\inv ) R(xu) R(v) L(vx) L(u).
\end{eqnarray*}
Apply both sides to $x\inv v\inv = (vx)\inv$:
\begin{eqnarray*}
[u \cdot x\inv v\inv x\inv ] R(x) R(v) &=& [(vx)\inv \cdot (vx)\inv (xu)\inv] R(xu) R(v) L(vx) L(u) \\
&=& [(vx)^{-2}] R((xu)\inv ) R(xu) R(v) L(vx) L(u) \\
&=& [(x\inv v\inv )^2 v] L(vx) L(u) \\
&=& x\inv  L((vx)\inv ) L(vx) L(u) \\
&=& u R(x\inv )
\end{eqnarray*}
Now replace $x$ with $x\inv$, $v$ with $v\inv$, and rearrange to get
\[
u R(xvx) = u R(x) R(v) R(x) ,
\]
which is the Moufang identity (N1).
\end{proof}

\section{Middle Bol loops}
\label{sec:mbol}

Since left Bol loops are universal for both \LIP\, and \LALT, it is natural to guess that the variety
of middle Bol loops (universal for \AAIP) is related to the variety of UF loops, perhaps in the sense that one variety is contained
in the other \cite{Sh3,Sh4,Syr1,Syr2}. It is true that small middle Bol loops are UF and vice versa. However, it turns out
that neither variety is contained in the other.

\begin{example}
This example shows that a middle Bol loop need not be flexible and a flexible middle Bol loop need not be SUF.
Table~\ref{table:16} is a flexible, middle Bol loop such that the isotope given by $x \circ y := x(8\ldv y)$ is not flexible. In
particular, $2\circ (0\circ 2) \neq (2\circ 0) \circ 2$.

\begin{table}[htb]
{ 
\footnotesize
\arraycolsep=1.2pt
\[
\begin{array}{c|cccccccccccccccc|}
\cdot &  0& 1& 2& 3& 4& 5& 6& 7& 8& 9&10&11&12&13&14&15 \\
\hline
  0 &  0& 1& 2& 3& 4& 5& 6& 7& 8& 9&10&11&12&13&14&15 \\
  1 &  1& 0& 3& 2& 5& 4& 7& 6& 9& 8&11& 10&13&12&15&14 \\
  2 &  2& 3& 0& 1& 6& 7& 4& 5&10&11& 8& 9&14&15&12&13 \\
  3 &  3& 2& 1& 0& 7& 6& 5& 4&11& 10& 9& 8&15&14&13&12 \\
  4 &  4& 5& 6& 7& 0& 1& 2& 3&12&13&14&15& 8& 9&10&11 \\
  5 &  5& 4& 7& 6& 1& 0& 3& 2&13&12&15&14& 9& 8&11&10 \\
  6 &  6& 7& 4& 5& 2& 3& 0& 1&14&15&12&13&10&11& 8& 9 \\
  7 &  7& 6& 5& 4& 3& 2& 1& 0&15&14&13&12&11&10& 9& 8 \\
  8 &  8& 9&10&11&12&13&14&15& 0& 1& 4& 5& 2& 3& 7& 6 \\
  9 &  9& 8&11&10&13&12&15&14& 1& 0& 5& 4& 3& 2& 6& 7 \\
 10 & 10&11& 8& 9&14&15&12&13& 4& 5& 0& 1& 7& 6& 2& 3 \\
 11 & 11&10& 9& 8&15&14&13&12& 5& 4& 1& 0& 6& 7& 3& 2 \\
 12 & 12&13&14&15& 8& 9&10&11& 2& 3& 7& 6& 0& 1& 4& 5 \\
 13 & 13&12&15&14& 9& 8&11&10& 3& 2& 6& 7& 1& 0& 5& 4 \\
 14 & 14&15&12&13&10&11& 8& 9& 7& 6& 2& 3& 4& 5& 0& 1 \\
 15 & 15&14&13&12&11&10& 9& 8& 6& 7& 3& 2& 5& 4& 1& 0 \\
\hline
\end{array}
\]
}  
\caption{A flexible, middle Bol loop with a nonflexible isotope}
\label{table:16}
\end{table}
\end{example}

\begin{example} The following construction slightly generalizes an example of Shelekhov \cite{Sh2, Sh4}.
Let $V$ be a vector space over a field of characteristic not 2, and let
$\phi : V\to \mathrm{Aut}(V\wedge V)$ be a homomorphism from the additive group of $V$
into the automorphism group of the wedge product $V\wedge V$.
On $W := V \oplus (V\wedge V)$, define
\[
(x , \omega) \cdot (y , \eta) := (x + y,\ \omega + \phi_{-x}(\eta ) + \phi_{-2x}(x\wedge y))
\]
Then $(W, \cdot)$ is a universally flexible loop, but this loop does not, in general, have the
AAIP. (Shelekhov's example is the case where $V = \mathbb{R}^2$,
$V\wedge V \cong \mathbb{R}$, and $\phi_x = e^x$.)
\end{example}

\begin{example}\label{ex:rb}
This example was found by using central extensions as described later.
Table \ref{suf-notMBol} shows a commutative UF loop which does not have the AAIP. Up to isomorphism,
this loop has only two loop isotopes. The other one, given by $x\circ y = x(3\ldv y)$ for example,
has the AAIP, but is neither middle Bol nor commutative.

This loop is also minimal in the sense that all of its proper subloops and all of its factor loops by
nontrivial normal subloops are middle Bol loops.

Interestingly, both this loop and its other isotope satisfy the semi-automorphic inverse property $(xyx)^{-1} = x^{-1} y^{-1} x^{-1}$.
Thus this property is universal in this loop.

\begin{table}[ht]
  \resizebox{\textwidth}{!}{%
  \begin{tabular}{c|cccccccccccccccccccccccccccccccc|}
  $\cdot$ & 1& 2& 3& 4& 5& 6& 7& 8& 9& 10& 11& 12& 13& 14& 15& 16& 17& 18& 19& 20& 21& 22& 23& 24& 25& 26& 27& 28& 29& 30& 31& 32 \\
  \hline\hline
  1& 1&  2&  3&  4&  5&  6&  7&  8&  9& 10& 11& 12& 13& 14& 15& 16& 17& 18& 19& 20& 21& 22& 23& 24& 25& 26& 27& 28& 29& 30& 31& 32 \\
  2& 2&  1&  4&  3&  6&  5&  8&  7& 10&  9& 12& 11& 14& 13& 16& 15& 18& 17& 20& 19& 22& 21& 24& 23& 26& 25& 28& 27& 30& 29& 32& 31 \\
  3& 3&  4&  7&  8&  9& 10& 11& 12& 13& 14&  1&  2& 15& 16&  5&  6& 27& 28& 23& 24& 31& 32& 19& 20& 29& 30& 17& 18& 25& 26& 21& 22 \\
  4& 4&  3&  8&  7& 10&  9& 12& 11& 14& 13&  2&  1& 16& 15&  6&  5& 28& 27& 24& 23& 32& 31& 20& 19& 30& 29& 18& 17& 26& 25& 22& 21 \\
  5& 5&  6&  9& 10&  1&  2& 14& 13&  3&  4& 16& 15&  8&  7& 12& 11& 21& 22& 25& 26& 17& 18& 29& 30& 19& 20& 31& 32& 23& 24& 27& 28 \\
  6& 6&  5& 10&  9&  2&  1& 13& 14&  4&  3& 15& 16&  7&  8& 11& 12& 22& 21& 26& 25& 18& 17& 30& 29& 20& 19& 32& 31& 24& 23& 28& 27 \\
  7& 7&  8& 11& 12& 14& 13&  1&  2& 16& 15&  3&  4&  6&  5& 10&  9& 23& 24& 27& 28& 29& 30& 17& 18& 31& 32& 19& 20& 21& 22& 25& 26 \\
  8& 8&  7& 12& 11& 13& 14&  2&  1& 15& 16&  4&  3&  5&  6&  9& 10& 24& 23& 28& 27& 30& 29& 18& 17& 32& 31& 20& 19& 22& 21& 26& 25 \\
  9& 9& 10& 13& 14&  3&  4& 16& 15&  7&  8&  6&  5& 12& 11&  2&  1& 31& 32& 29& 30& 27& 28& 25& 26& 23& 24& 21& 22& 19& 20& 17& 18 \\
  10& 10&  9& 14& 13&  4&  3& 15& 16&  8&  7&  5&  6& 11& 12&  1&  2& 32& 31& 30& 29& 28& 27& 26& 25& 24& 23& 22& 21& 20& 19& 18& 17 \\
  11& 11& 12&  1&  2& 16& 15&  3&  4&  6&  5&  7&  8& 10&  9& 14& 13& 19& 20& 17& 18& 25& 26& 27& 28& 21& 22& 23& 24& 31& 32& 29& 30 \\
  12& 12& 11&  2&  1& 15& 16&  4&  3&  5&  6&  8&  7&  9& 10& 13& 14& 20& 19& 18& 17& 26& 25& 28& 27& 22& 21& 24& 23& 32& 31& 30& 29 \\
  13& 13& 14& 15& 16&  8&  7&  6&  5& 12& 11& 10&  9&  1&  2&  3&  4& 30& 29& 32& 31& 24& 23& 22& 21& 28& 27& 26& 25& 18& 17& 20& 19 \\
  14& 14& 13& 16& 15&  7&  8&  5&  6& 11& 12&  9& 10&  2&  1&  4&  3& 29& 30& 31& 32& 23& 24& 21& 22& 27& 28& 25& 26& 17& 18& 19& 20 \\
  15& 15& 16&  5&  6& 12& 11& 10&  9&  2&  1& 14& 13&  3&  4&  7&  8& 26& 25& 22& 21& 20& 19& 32& 31& 18& 17& 30& 29& 28& 27& 24& 23 \\
  16& 16& 15&  6&  5& 11& 12&  9& 10&  1&  2& 13& 14&  4&  3&  8&  7& 25& 26& 21& 22& 19& 20& 31& 32& 17& 18& 29& 30& 27& 28& 23& 24 \\
  17& 17& 18& 27& 28& 21& 22& 23& 24& 31& 32& 19& 20& 30& 29& 26& 25& 13& 14& 11& 12&  7&  8&  5&  6& 16& 15&  3&  4&  2&  1&  9& 10 \\
  18& 18& 17& 28& 27& 22& 21& 24& 23& 32& 31& 20& 19& 29& 30& 25& 26& 14& 13& 12& 11&  8&  7&  6&  5& 15& 16&  4&  3&  1&  2& 10&  9 \\
  19& 19& 20& 23& 24& 25& 26& 27& 28& 29& 30& 17& 18& 32& 31& 22& 21& 11& 12&  5&  6& 16& 15&  3&  4&  1&  2& 13& 14&  9& 10&  8&  7 \\
  20& 20& 19& 24& 23& 26& 25& 28& 27& 30& 29& 18& 17& 31& 32& 21& 22& 12& 11&  6&  5& 15& 16&  4&  3&  2&  1& 14& 13& 10&  9&  7&  8 \\
  21& 21& 22& 31& 32& 17& 18& 29& 30& 27& 28& 25& 26& 24& 23& 20& 19&  7&  8& 16& 15& 13& 14&  2&  1& 11& 12&  9& 10&  5&  6&  3&  4 \\
  22& 22& 21& 32& 31& 18& 17& 30& 29& 28& 27& 26& 25& 23& 24& 19& 20&  8&  7& 15& 16& 14& 13&  1&  2& 12& 11& 10&  9&  6&  5&  4&  3 \\
  23& 23& 24& 19& 20& 29& 30& 17& 18& 25& 26& 27& 28& 22& 21& 32& 31&  5&  6&  3&  4&  2&  1& 13& 14&  9& 10& 11& 12&  7&  8& 16& 15 \\
  24& 24& 23& 20& 19& 30& 29& 18& 17& 26& 25& 28& 27& 21& 22& 31& 32&  6&  5&  4&  3&  1&  2& 14& 13& 10&  9& 12& 11&  8&  7& 15& 16 \\
  25& 25& 26& 29& 30& 19& 20& 31& 32& 23& 24& 21& 22& 28& 27& 18& 17& 16& 15&  1&  2& 11& 12&  9& 10&  5&  6&  8&  7&  3&  4& 13& 14 \\
  26& 26& 25& 30& 29& 20& 19& 32& 31& 24& 23& 22& 21& 27& 28& 17& 18& 15& 16&  2&  1& 12& 11& 10&  9&  6&  5&  7&  8&  4&  3& 14& 13 \\
  27& 27& 28& 17& 18& 31& 32& 19& 20& 21& 22& 23& 24& 26& 25& 30& 29&  3&  4& 13& 14&  9& 10& 11& 12&  8&  7&  5&  6& 16& 15&  1&  2 \\
  28& 28& 27& 18& 17& 32& 31& 20& 19& 22& 21& 24& 23& 25& 26& 29& 30&  4&  3& 14& 13& 10&  9& 12& 11&  7&  8&  6&  5& 15& 16&  2&  1 \\
  29& 29& 30& 25& 26& 23& 24& 21& 22& 19& 20& 31& 32& 18& 17& 28& 27&  2&  1&  9& 10&  5&  6&  7&  8&  3&  4& 16& 15& 13& 14& 11& 12 \\
  30& 30& 29& 26& 25& 24& 23& 22& 21& 20& 19& 32& 31& 17& 18& 27& 28&  1&  2& 10&  9&  6&  5&  8&  7&  4&  3& 15& 16& 14& 13& 12& 11 \\
  31& 31& 32& 21& 22& 27& 28& 25& 26& 17& 18& 29& 30& 20& 19& 24& 23&  9& 10&  8&  7&  3&  4& 16& 15& 13& 14&  1&  2& 11& 12&  5&  6 \\
  32& 32& 31& 22& 21& 28& 27& 26& 25& 18& 17& 30& 29& 19& 20& 23& 24& 10&  9&  7&  8&  4&  3& 15& 16& 14& 13&  2&  1& 12& 11&  6&  5 \\
  \hline
  \end{tabular}}
  \caption{A UF loop which does not satisfy \AAIP}
  \label{suf-notMBol}
\end{table}
\end{example}

\section{Diassociativity}
\label{sec:diassoc}

To prove that a loop is diassociative, it is natural to introduce the following definition:

\begin{definition}
Given elements $x_1,\ldots , x_n$ in a loop,
let $A(x_1,\ldots , x_n)$ abbreviate the statement that
$x_1,\ldots , x_n$ \emph{associates}, in the sense that all associated
products of the elements in that order are equal.
\end{definition}

The number of such products is the Catalan number $C_{n-1} = \frac{(2n-2)!}{n!(n-1)!}$.
For example, when $n=4$, $C_{4-1} = 5$, and $A(w,x,y,z)$ iff
\[
w \cdot (x \cdot y z) =
w \cdot (x y\cdot  z) =
w x \cdot y z =
(w \cdot x y)\cdot z =
(w x \cdot y)\cdot z  \ \ .
\]

\begin{definition}
In an IP loop, let $H(n)$ abbreviate the statement that for all $a,b$ in the loop,
$A(x_1, \ldots , x_n)$ holds for all $x_1, \ldots, x_n \in \{1, a, a\inv, b, b\inv\}$.
\end{definition}

\begin{lemma}
In any IP loop $Q$:
\begin{itemize}
\item
$H(n) \to H(m)$ whenever $m \le n$.
\item
$Q$ is diassociative iff $H(n)$ holds for all $n$.
\item
$H(3)$ holds iff $Q$ is flexible and alternative.
\end{itemize}
\end{lemma}

Moufang's original proof \cite{M1,M2} of what we now call Moufang's Theorem was inductive.
Moufang loops satisfy a number of additional properties which allow for a proof of
Moufang's Theorem without an explict induction; this is the approach taken in \cite{Br}.
More recently, Dr\'{a}pal found a simpler proof in the spirit of Moufang's original proof \cite{Dr}.

Our proof of diassociativity will prove $H(n)$ by induction on $n$. Definition \ref{def:step} below introduces
a class of loops in which one can step up easily from $H(n)$ to $H(n+1)$.
This class includes all SUF IP loops, and hence all Moufang loops.

\begin{definition}
If $X$ is a subset of an IP loop $Q$, then $X$ \emph{strongly associates}
iff $A(x,y,z)$ for all $x,y,z \in X \cup X\inv$.
Let $\SA(x,y,z)$ abbreviate the statement that $\{x,y,z\}$ strongly associates.
\end{definition}

In Moufang loops, $A(x,y,z)$ is equivalent to $SA(x,y,z)$.

The notation $\SA(x,y,z)$ does not imply that $x,y,z$ are distinct.
If $Q$ is also alternative and flexible, then every $\{a,b\}$ strongly associates, and the assertion $\SA(a,b,c)$
is naturally written as $6 \cdot 8 = 48$ equations, since there are $6$ permutatons
of $\{a,b,c\}$ and then $2^3 = 8$ choices of which of $a,b,c$ to invert.
However, by IP, these $48$ come in $24$ equivalent pairs,
since $A(x,y,z) \leftrightarrow A(z\inv, y\inv, x\inv)$;
for example, $A(a,b\inv,c\inv)$ is equivalent to $A(c,b,a\inv)$.

\begin{definition}\label{def:step}
A \emph{stepping-up loop} is an alternative, flexible IP loop $Q$ satisfying the implication
$\SA(x,y,z) \to \SA(x,y,yz)$.
\end{definition}

Note that the property  $\SA(x,y,z)$ is invariant under permutations of $x,y,z$ and under replacing any of $x,y,z$ by its inverse.
Thus in a stepping-up loop, if $\SA(x,y,z)$ holds, then also $\SA(x\inv,y\inv, y\inv z\inv)$ , $\SA(x,y,zy)$ , $\SA(x, xy,z)$,
$\SA(x, x\inv y, z)$, etc.

Stepping up loops form a quasivariety because the implication $\SA(x,y,z) \to \SA(x,y,yz)$ can be written as a collection of
quasi-identities. Moufang loops are stepping up; this is a consequence of Moufang's Theorem or can be shown directly.
There exist some Steiner loops which are stepping up, but not all of them are.

\begin{lemma}
Every stepping-up loop is diassociative.
\end{lemma}
\begin{proof}
We shall prove by induction that $H(n)$ holds for all $n$.
We already have $H(3)$ by the IP, \FLEX, \LALT\, and \RALT.
Now, we fix $n \ge 3$ and assume $H(n)$; we shall prove $H(n+1)$,
so fix $a,b$, and fix $x_0, x_1, \ldots, x_n \in \{1, a, a\inv, b, b\inv\}$;
we shall prove $A(x_0, x_1, \ldots , x_n)$.
Note that  $H(n)$ implies
$\SA(x_1 \cdots x_j ,\; y_1 \cdots y_k ,\; z_1 \cdots z_\ell )$
whenever $j + k + \ell \le n$ and all
$x_1, \ldots, x_j , y_1, \ldots, y_k ,\allowbreak z_1, \ldots, z_\ell
\in \{1, a, a\inv, b, b\inv\}$.

In view of $H(n)$, there are $n$, not $C_n$, products that we must prove
are equal, namely the
$P_\ell := (x_0 \cdots x_\ell) \cdot (x_{\ell+1} \cdots x_n)$
for $0 \le \ell < n$.

$P_0 = P_1$ and $P_{n-2} = P_{n-1}$: To prove $P_0 = P_1$, Without loss of generality,
$x_0 = a$, so we must prove
$a \cdot (x_1 \cdots x_n) = (a_0 x_1) \cdot (x_2 \cdots x_n)$.
This is clear from ALT+IP if $x_1 \in \{1,a,a\inv\}$,
so Without loss of generality, $x_1 = b$.  Then the result follows from
$\SA(a, b,\; x_2 \cdots x_n)$, which in turn follows from
$\SA(a, b, 1)$ and the stepping-up property.

Now, we prove by induction on $\ell$ that $P_\ell = P_0$,
so we fix $\ell$, assume that $P_0 = \cdots = P_{\ell -1}$,
and we prove $P_\ell = P_0$.  We may assume that $2 \le \ell \le n-2$.
Without loss of generality, $x_\ell = a$.  We first attempt to prove
$P_\ell = P_{\ell -1}$; that is,
\[
(x_0 \cdots x_{\ell -1} a) \cdot (x_{\ell+1} \cdots x_n) =
(x_0 \cdots x_{\ell -1}) \cdot (a x_{\ell+1} \cdots x_n) \ \ .
\]
There are three cases.

\noindent Case 1: at least one of $x_0,x_{\ell -1}, x_{\ell+1}, x_n$ is $1$;
then just apply $H(n)$.

\noindent Case 2: at least one of $x_0,x_{\ell -1}, x_{\ell+1}, x_n$ is
$a$ or $a \inv$; then we can verify
$\SA(x_0 \cdots x_{\ell -1} ,\; a ,\; x_{\ell+1} \cdots x_n)$;
for example, if $ x_{\ell+1} = a\inv$, we have
$\SA(x_0 \cdots x_{\ell -1} ,\; a ,\; x_{\ell+2} \cdots x_n)$ from $H(n)$,
and then $\SA(x_0 \cdots x_{\ell -1} ,\; a ,\; x_{\ell+1} \cdots x_n)$
by the stepping-up property.

\noindent Case 3: assume that
$x_0,x_{\ell -1}, x_{\ell+1}, x_n \in \{b, b\inv\}$.
Now, we prove instead $P_\ell = P_0$.  Without loss of generality, $x_0 = b$, so we must prove
\[
(b x_1 \cdots x_\ell) \cdot (x_{\ell+1} \cdots x_n) =
b \cdot  (x_1 \cdots x_\ell x_{\ell+1} \cdots x_n) \ \ .
\]
But now we use $\SA(b,\; x_1 \cdots x_\ell,\; x_{\ell+1} \cdots x_n)$,
which follows
(using $x_{\ell+1} \in \{b, b\inv\}$)
from $\SA(b,\; x_1 \cdots x_\ell,\; x_{\ell+2} \cdots x_n)$,
and this in turn follows from $H(n)$.
\end{proof}

\begin{lemma}
Every SUF IP loop is a stepping-up loop.
\end{lemma}
\begin{proof}
Assume $\SA(a,b,c)$.  We need to show $\SA(a,b,bc)$.
We can write out $\SA(a,b,bc)$ explicitly as 48 equations, but,
as remarked above, the IP reduces this to 24, since we can just
write the ones for which $bc$ is not inverted.  But also, the
hypotheses are symmetric in $a / a\inv$ (since
$\SA(a,b,c) \leftrightarrow \SA(a\inv, b,c)$), so in verifying
$\SA(a,b,bc)$, we we need only write the equations in
which $a$ is not inverted.  Thus, we need to verify the following twelve:

\setlength{\columnsep}{-40pt}
\begin{multicols}{2}
\begin{enumerate}
\renewcommand{\labelenumi}{\theenumi.}  
\item\label{sa:1}
$A(a, b, bc): $
\ $ ab \cdot  bc = a \cdot  b bc.       $
\item\label{sa:5}
$A(b, a, bc): $
\ $ ba \cdot  bc = b \cdot  a bc.       $
\item\label{sa:2}
$A(a, bc, b): $
\ $ a bc \cdot b = a \cdot  bcb.        $
\item\label{sa:3}
$A(b, bc, a): $
\ $ b bc \cdot a = b \cdot  bca.        $
\item\label{sa:6}
$A(bc, a, b): $
\ $ bca \cdot b =  bc \cdot  ab.        $
\item\label{sa:4}
$A(bc, b, a): $
\ $ bcb \cdot a =  bc \cdot  ba.        $
\item\label{sa:12}
$A(a, b\inv, bc): $
\ $ ab\inv  \cdot  bc = a \cdot  b\inv  bc. $
\item\label{sa:8}
$A(b\inv, a, bc): $
\ $ b\inv a \cdot  bc = b\inv  \cdot  a bc. $
\item\label{sa:10}
$A(a, bc, b\inv): $
\ $ a bc \cdot b\inv  = a \cdot  bcb\inv .  $
\item\label{sa:7}
$A(b\inv, bc, a): $
\ $ b\inv  b c \cdot a = b\inv  \cdot  bca.  $
\item\label{sa:9}
$A(bc, a, b\inv): $
\ $ bca \cdot b\inv  =  bc \cdot  ab\inv .  $
\item\label{sa:11}
$A(bc, b\inv, a): $
\ $ bcb\inv  \cdot a =  bc \cdot  b\inv a.  $
\end{enumerate}
\end{multicols}
Now, \pref{sa:7} is immediate from the IP.
To prove \pref{sa:5}, apply $L(b)R(c\inv) R(bc) = R(c\inv) R(bc)L(b)$ to $ac$.
To prove \pref{sa:8},
apply $R(bc) = R(c)L(b)R(c\inv)R(bc)L(b\inv)$ to $b\inv a$.

Observe that once we have proved an equation, symmetry of the
hypotheses yields a number of other associations.
For example, given \pref{sa:8},  $A(b\inv, a, bc)$,
we also have $A(c\inv, a, cb)$ (interchanging $b,c$), and then
mirror symmetry yields $A(bc, a, c\inv)$.
More formally, the justification for this symmetry is that \pref{sa:8} really
establishes $\forall x,y,z \,[ \SA(x,y,z) \to A(y\inv, x, yz)]$.
Thus, we get $\SA(a,b,c) \leftrightarrow \SA(a,c,b) \to A(c\inv, a, cb)$,
and we likewise get $\SA(a,b,c) \leftrightarrow \SA(a\inv,c\inv,b\inv)
\to  A(c, a\inv , c\inv b\inv ) \leftrightarrow A(bc, a, c\inv) $.

To prove \pref{sa:6}, apply $L(bc)R(c)R(b) = R(c)R(b)L(bc)$ to $a c\inv$.
The right side is $bc \cdot ab$.
The left side is $((bc \cdot a c\inv)\cdot c ) \cdot b = bca \cdot b$
by $A(bc, a, c\inv)$.

We next prove \pref{sa:4}.  First note that
\[
b c a\inv \cdot b a =
b \cdot (c a\inv \cdot ba) =
b \cdot (c a\inv b \cdot a) =
(b c \cdot a\inv b) \cdot a =
(bc a\inv \cdot b) \cdot a\,.
\]
For the first `=', apply $[L(b), R(a\inv) R(ba)] = I$ to $c$.
For the second `=', use $A(c a\inv, b, a)$, which follows from
\pref{sa:8} by symmetry.
For the third `=', apply $[L(b), R(a\inv b) R(a)] = I$ to $c$.
For the fourth `=', use $A(bc, a\inv, b)$, which follows from
\pref{sa:6}.  Now apply $L(bc) = L(bc a\inv) R(a\inv) L(a c\inv b\inv) L(bc) R(a)$
to $ba$.
The left side is $bc \cdot ba$.  Since
$(ba) L(bc a\inv) = (bc a\inv \cdot b) \cdot a$,
the right side is $bcb \cdot a$.

To prove \pref{sa:2}, note that $abc \cdot b = ab\cdot cb = a \cdot bcb$.
We are using $A(ab,c,b)$ and $A(a,b,cb)$, which follow, by symmetry
from \pref{sa:5} and \pref{sa:4}, respectively.

To prove \pref{sa:3}:  Note that $b \cdot a\inv cb = b a\inv  \cdot cb$ (i.e.,
$A(b, a\inv, cb )$) follows from \pref{sa:6} and symmetry.
Thus, applying $L(b) R(b\inv c\inv) R(bcb) =  R(b\inv c\inv) R(bcb) L(b)$
to $a\inv cb$, we get
$ (b a\inv) \cdot (bcb) = b \cdot (a\inv \cdot bcb) $, and hence,
\[
(a b\inv) \cdot ( b \cdot (a\inv \cdot bcb))  = bcb\,.
\]
Note that $a\inv \cdot bcb = a\inv b c \cdot b$ (i.e., $A(a\inv, bc, b)$)
by \pref{sa:2} and symmetry.
Apply $L(a b\inv) R(b\inv)R(a) = R(b\inv)R(a) L(a b\inv)$ to
$b \cdot (a\inv \cdot bcb) = b \cdot (a\inv bc) \cdot b$.
The left side is $(bcb) R(b\inv)R(a) = bca$.
The right side is
$
(a b\inv) \cdot ( ( b \cdot (a\inv bc) ) \cdot a) =
(a b\inv) \cdot ((b a\inv b \cdot c) \cdot a)
$
since $A(b, a\inv b, c)$ holds by \pref{sa:2} and symmetry.
So, we get:
\[
(a b\inv) \cdot ((b a\inv b \cdot c) \cdot a) = bca\,.
\]
Finally, observe that $b a\inv b \cdot c = ba\inv \cdot bc$
(i.e., $A(b a\inv, b,c)$)
by \pref{sa:4} and symmetry.
Apply $R(a) L(a b\inv)L(b) = L(a b\inv)L(b) R(a)$ to $b a\inv b \cdot c$.
The right side is $bbc \cdot a$.
Since $(b a\inv b \cdot c) R(a) L(a b\inv) = bca$, the left side
is $b \cdot bca$.

We next verify \pref{sa:11}:  First observe that
\[
a \cdot (b\inv \cdot c\inv b\inv a) =
a \cdot b\inv c\inv b\inv \cdot a =
(a b\inv c\inv \cdot b\inv) \cdot a =
a b\inv \cdot c\inv b\inv a\,.
\]
The first two `='s use flexibility plus
$A(b\inv, c\inv b\inv, a)$ and $A(a, b\inv c\inv, b\inv)$,
which are symmetric variants of \pref{sa:2}.
The third applies $[L(a b\inv), R(b\inv)R(a)] = I$ to $c\inv$.
Next we get
\[
[ bc \cdot (b\inv \cdot c\inv b \inv a) ] \cdot (a\inv b c) = b c b\inv
\]
by applying $L(a)R(a\inv b c) L(a\inv) L(bc) = L(bc) R(a\inv bc) $ to
$b\inv \cdot c\inv b\inv a$.
The left side is
$ ( a b\inv \cdot c\inv b\inv a) R(a\inv b c) L(a\inv) L(bc) = b c b\inv $.
Finally,
\begin{multline*}
 bc \cdot  b\inv a = (b\inv )  R(a) L(bc) =
(b\inv) R(c\inv b\inv a) L(bc) R(a\inv b c) R(a) =  \\
(b c b\inv) R(a) = (b c b\inv) \cdot a\,.
\end{multline*}

To prove \pref{sa:10},
note that $a b c \cdot b\inv = ab \cdot cb\inv = a \cdot b c b\inv$,
using $A(ab, c, b\inv)$ and $A(a, b,c b\inv)$,
which follow, by symmetry from \pref{sa:8} and \pref{sa:11}, respectively.

To prove \pref{sa:9}, apply $L(b)R(c)L(b\inv)L(bc) = L(bc) R(c)$
to $a b\inv c\inv$.  Using the fact that $A(b, ab\inv, c)$ (from \pref{sa:10}
and symmetry), the left side is $bc \cdot  ab\inv$.
The right side is $ (bc \cdot a b\inv c\inv) \cdot c = bca \cdot b\inv$, since
\[
bc \cdot a b\inv c\inv  =
a R(b\inv)R(c\inv) L(bc)  = a L(bc) R(b\inv)R(c\inv)   = (bca \cdot b\inv) \cdot c\inv\,.
\]

To prove \pref{sa:12}, we need to show $ ab\inv  \cdot  bc = a c$.
We begin by showing
\[
c\inv b \cdot c a b\inv = (c\inv \cdot  bca) \cdot b\inv =
c\inv \cdot (b \cdot c a b\inv)\,.
\]
For the first `=',
apply $L(c\inv)R(b\inv)L(c)L(c\inv b) = L(c\inv b) R(b\inv)$ to $ca$.
The right side is $(c\inv b \cdot ca) \cdot b\inv = (c\inv \cdot  bca) \cdot b\inv$,
using $A(c\inv, b, ca)$, which follows from  \pref{sa:8} and symmetry.
The left side is $c\inv b \cdot cab\inv$.
For the second `=',
apply $L(b c) L(c\inv) = R(b)L(bc) L(c\inv)R(b\inv)$ to $a b\inv$.
The left side is
$c\inv  \cdot (bc \cdot a b\inv) = c\inv \cdot (b \cdot c a b\inv) $,
using $A(b,c,ab\inv)$, which follows from  \pref{sa:9} and symmetry.
The right side is $(c\inv \cdot  bca) \cdot b\inv$.
We shall now get
\[
( ab\inv  \cdot  bc) \cdot a b\inv =
a b\inv \cdot (b \cdot c a b\inv) = a \cdot c  a b\inv = ac \cdot a b\inv\,,
\]
and \pref{sa:12} will follow by cancelling the $a b\inv$.
To prove the second `=', apply
$R(c a b\inv) L(c)L(ab\inv) = L(c)L(ab\inv) R(c a b\inv)$ to $c\inv b$.
The right side is $ a \cdot c  a b\inv$.
Since $(c\inv b) R(c a b\inv) = c\inv \cdot (b \cdot c a b\inv) $,
the left side is $a b\inv \cdot (b \cdot c a b\inv)$.
The first `=' follows by FLEX and $A(b,c,a b\inv)$, which follows
from \pref{sa:9} and symmetry.  The third `=' follows
from $A(a,c,ab\inv)$, which follows from \pref{sa:5} and symmetry.

To verify \pref{sa:1}, start with
\[
a  c\inv a \cdot b\inv  =  a \cdot c\inv a b\inv =
(a c\inv b\inv  \cdot ba) \cdot b\inv\,.
\]
The first `=' is by $A(a, c\inv a, b\inv)$ (using \pref{sa:2} and symmetry).
For the second `=',
apply $R(ba)R(b\inv) L(a) = L(a) R(ba)R(b\inv)$ to $c\inv b\inv$,
using $c\inv b\inv \cdot ba = c\inv a$ (by \pref{sa:12} and symmetry)
to evaluate the left side.
Now cancel the $b\inv$ and divide by $a c\inv b\inv$ to get
$b c a\inv \cdot a  c\inv a  = ba$, so that $(ac\inv) R(a)L(bca\inv)  = ba$.
Then, applying
$R(bc)L(bc a\inv) = R(a)L(bca\inv)R(a\inv) R(bc)$ to $a c\inv$ yields:
\[
bca\inv \cdot (ac\inv \cdot bc) = bbc\,.
\]
Now, apply $L(bca\inv)L(a)R(bc) = R(bc)L(bca\inv)L(a)$ to $a c\inv$,
to get
\[
ab \cdot bc  =
a \cdot (b c a\inv \cdot  ( a c\inv \cdot bc)) = a \cdot bbc\,.
\]
\end{proof}

Theorem \ref{thm:main2} now follows:

\begin{corollary}
Every SUF IP loop is diassociative.
\end{corollary}

\section{Open Problems}
\label{sec:problems}

The main open problem here is the following:

\begin{problem}\label{prb:prob1}
Does there exist an SUF IP loop which is not a Moufang loop?
\end{problem}

We have spent a lot of effort over a period of many years on both sides of this question. Using \textsc{Prover9},
we know various cases in which SUF IP loops are Moufang. For example, every commutative SUF IP loop is Moufang, as is
every nilpotent SUF IP loop of nilpotency class $2$. We are not particularly motivated to present the
proofs of these claims, because they are long ``equation-bashing'' proofs and not very enlightening.
None of the proofs we have in special cases seem to give any insight into the general question.

In the negative direction, naive \textsc{Mace4} searches up to order 20 have been unsuccessful.
As part of their PhD dissertation work \cite{RB}, one of us (RB) attacked the problem using central extensions,
the same methodology used to construct Example \ref{ex:rb}.
They checked central extensions of the Moufang loops of orders 16, 32, 64 and 81 by cyclic groups of various
orders; see \cite{RB} for details. Unfortunately, none of those cases were successful.

Despite all this, it is still our opinion that proper (i.e. nonMoufang) SUF IP loops exist.

By Theorem~\ref{thm:main1}, an affirmative answer to Problem \ref{prb:prob1} will also be an affirmative answer to the following.

\begin{problem}\label{prb:prob2}
Does there exist an SUF loop which is not a UF loop?
\end{problem}

Here we have not searched as extensively as for Problem \ref{prb:prob1}.

We have mentioned the property in loops that the inner mappings $R(x)L(x)\inv$ are all automorphisms; this property
implies flexibility.

\begin{problem}
Does there exist a loop satisfying the aforementioned property universally, but which does not satisfy the AAIP?
\end{problem}

Example \ref{ex:rb} also suggests the following problems.

\begin{problem}
Does there exist a UF loop which does not satisfy the semi-automorphic inverse property $(xyx)\inv = x\inv y\inv x\inv$?
\end{problem}

\begin{problem}
Does there exist a commutative UF loop with the AAIP which is not a middle Bol loop?
\end{problem}

\begin{ack}
We would like to thank Petr Vojt\v{e}chovsk\'{y} for providing us with GAP code for computing central extensions.

Our investigations were aided by the automated reasoning tool \textsc{Prover9} and the
finite model builder \textsc{Mace4} \cite{MC}, as well as the LOOPS package \cite{NV} for GAP \cite{GAP4}.
\end{ack}


\end{document}